\documentclass[12pt,reqno]{amsart}
\usepackage[utf8]{inputenc}

\usepackage{amsfonts, amssymb, amsmath, amsthm, mathtools}
\usepackage{mathrsfs} % \mathscr
\usepackage{latexsym}% utilisation des symboles LaTeX pour avoir un beau LaTeX
\usepackage{enumerate}
\usepackage{multicol}
\usepackage{times}% font
\usepackage{verbatim}
\usepackage{tikz-cd}
\usepackage{fullpage}
\usepackage{here}
\usepackage{url}
\usepackage{csquotes}
\usepackage{placeins}
\usepackage{epic}
\usepackage{graphicx}
\usepackage{epstopdf}
\usepackage{pstricks}
\usepackage{pst-plot}
\usepackage{tikz}
\usetikzlibrary{shapes.geometric, positioning, calc}

\input xypic
\xyoption{all}
\xyoption{poly}

\usepackage[colorlinks=true]{hyperref}
\hypersetup{citecolor=blue, linkcolor=blue}

\usepackage[capitalise, nameinlink]{cleveref}

\crefname{equation}{}{}
\crefname{figure}{{\sc Figure}}{{\sc Figure}}
\crefname{subsection}{Subsection}{Subsections}

\usetikzlibrary{matrix,arrows,decorations.pathmorphing}

\usepackage[euler]{textgreek}
\usepackage{tikz,tikz-cd}
\usepackage[colorinlistoftodos, textwidth = 2.3cm]{todonotes}
\newcommand{\nc}{\newcommand}

\newtheorem{theorem}{Theorem}[section]
\newtheorem{proposition}[theorem]{Proposition}
\newtheorem{lemma}[theorem]{Lemma}
\newtheorem{corollary}[theorem]{Corollary}

\newtheorem*{claim*}{Claim}

\theoremstyle{definition}

\newtheorem{remark}[theorem]{Remark}

\newcommand{\C}{{\mathbb C}}
\newcommand{\F}{{\mathbb F}}
\newcommand{\Z}{{\mathbb Z}}

\DeclarePairedDelimiter\floor{\lfloor}{\rfloor}
\DeclareMathOperator{\ord}{ord}
\usepackage{tabstackengine}
\stackMath

\numberwithin{equation}{section} 
\numberwithin{figure}{section}
\numberwithin{table}{section}

\nc{\Beaver}[1]{\todo[size=\tiny,color=cyan!10]{#1 \\ \hfill --- Beaver}}
\nc{\BEAVER}[1]{\todo[size=\tiny,inline,color=cyan!10]{#1
		\\ \hfill --- Beaver}}
  
\nc{\Semin}[1]{\todo[size=\tiny,color=magenta!10]{#1 \\ \hfill --- Semin}}
\nc{\SEMIN}[1]{\todo[size=\tiny,inline,color=magenta!10]{#1
		\\ \hfill --- Semin}}

\nc{\nt}[1]{\todo[size=\tiny,color=exgreen!10]{#1 \\ \hfill --- Note}}
\nc{\NT}[1]{\todo[size=\tiny,inline,color=exgreen!10]{#1
		\\ \hfill --- Note}}

\title{Explicit constructions of Diophantine tuples \\ over finite fields}

\author{Seoyoung Kim}
\address{Mathematisches Institut, Georg-August Universit\"at G\"ottingen, Bunsenstrasse 3-5, 37073 G\"ottingen, Germany}
\email{seoyoung.kim@mathematik.uni-goettingen.de}
\author{Chi Hoi Yip}
\address{Department of Mathematics \\ University of British Columbia \\ Vancouver  V6T 1Z2 \\ Canada}
\email{kyleyip@math.ubc.ca}
\author{Semin Yoo}
\address{Discrete Mathematics Group \\ Institute for Basic Science \\ 55 Expo-ro Yuseong-gu, Daejeon 34126 \\ South Korea}
\email{syoo19@ibs.re.kr}
\keywords{Diophantine tuples, character sum, cyclotomic polynomial \\ \indent Corresponding author: Semin Yoo (syoo19@kias.re.kr)}
\begin{document}

\maketitle

\begin{abstract}
A Diophantine $m$-tuple over a finite field $\F_q$ is a set $\{a_1,\ldots, a_m\}$ of $m$ distinct elements in $\mathbb{F}_{q}^{*}$ such that $a_{i}a_{j}+1$ is a square in $\F_q$ whenever $i\neq j$. In this paper, we study $M(q)$, the maximum size of a Diophantine tuple over $\F_q$, assuming the characteristic of $\F_q$ is fixed and $q \to \infty$. By explicit constructions, we improve the lower bound on $M(q)$. In particular, this improves a recent result of Dujella and Kazalicki by a multiplicative factor.
\end{abstract}

\section{Introduction}

A set of $m$ positive integers is a \textit{Diophantine $m$-tuple} if the product of any two distinct elements in the set is one less than a square.  The first known example of Diophantine $4$-tuples is $\{1,3,8,120\}$, due to Fermat. It had been conjectured that there is no Diophantine $5$-tuple and the conjecture was recently proved to be true in the sequel of striking works by Dujella \cite{duj2004diophantine}, and He, Togb\'e, and Ziegler \cite{HTZ19}. Diophantine $m$-tuples have been also generalized and studied in various algebraic domains, such as rational numbers, finite fields, and polynomial rings; we refer to Dujella's webpage \cite{Dujella-page} for a comprehensive list of references.

In this paper, we focus on Diophantine tuples over finite fields. Throughout the paper, let $p$ be an odd prime, $q$ a power of $p$, $\F_q$ the finite field with $q$ elements, and $\F_q^*=\F_q \setminus \{0\}$. A \emph{Diophantine $m$-tuple over $\F_q$} is a set $\{a_1,\ldots, a_m\}$ of $m$ distinct elements in $\mathbb{F}_{q}^{*}$ such that $a_{i}a_{j}+1$ is a square in $\F_q$ (including $0$) whenever $i\neq j$. For a prime $p$, Diophantine tuples over $\mathbb{F}_{p}$ were first formally introduced by Dujella and Kazalicki in \cite{21DK} and were studied by using their connections to elliptic curves and their modularity. However, they have been implicitly studied at least two decades ago, see for example \cite{D02, D04, G01}. In the present paper, we focus on the case of a finite extension of $\mathbb{F}_p$. Recently, a few papers \cite{21DK, k23, KYY, MRS21, Shparlinski23} are devoted to Diophantine tuples and their generalizations over finite fields. 

For each odd prime power $q$, let $M(q)$ be the maximum size of a Diophantine tuple in $\F_q$. Estimating $M(q)$ is the exact analog of estimating the maximum size of a Diophantine tuple over integers, and thus it is of special interest. When $q$ is a square, we have $M(q) \geq \sqrt{q}-1$ since $\F_{\sqrt{q}}^*$ is a Diophantine tuple in $\F_q$, and it is shown in \cite[Proposition 1.3]{KYY} that $M(q)\leq \sqrt{q}+\frac{5}{2}$ for all odd $q$. Thus, the value of $M(q)$ is essentially determined when $q$ is a square. On the other hand, it is not even clear what should be the right magnitude of $M(q)$ when $q$ is a non-square. 

Dujella and Kazalicki in \cite[Theorem 17]{21DK} showed that $M(p) \geq (1-o(1))\log_4 p$ as $p \to \infty$. In fact, the identical proof extends to all finite fields with odd characteristic. This lower bound follows from a standard application of Weil's bound. It is not hard to show a slightly stronger statement, namely $(1-o(1))\log_4 q$ is a lower bound for any \emph{maximal Diophantine tuple} $A$ over $\F_q$ (maximal in the sense that $A$ cannot be extended to be a Diophantine tuple with a larger size); see \cref{prop:maximal}. Thus, it may not be optimal to use this uniform lower bound for the size of maximal Diophantine tuples to deduce a lower bound on $M(q)$, and $(1-o(1))\log_4 q$ can be regarded as \emph{the trivial lower bound} on $M(q)$. \textcolor{black}{However, improving this trivial lower bound is not easy. We conjecture that $M(q)\geq (1-o(1))\log_2 q$ by assuming that the squares behave like a random set.}

We focus on explicit constructions of large Diophantine tuples and in particular give an improved lower bound on $M(q)$ for an infinite family of $q$. When $p$ is fixed, and $q=p^n \to \infty$, this improves the trivial lower bound $\log_4 q$ by a constant factor. 

\begin{theorem}\label{main}
If $q>7$ is a power of an odd prime $p$ and $q \equiv 1,5,7 \pmod 8$, then
\[M(q) \geq \left \lfloor Q \right \rfloor, \quad \text{where } Q=\frac{p}{p-1}\bigg(\frac{\frac{1}{2}\log q -2\log\log q}{\log 2} +1\bigg).
\]
In particular, if $p$ is fixed and $q \to \infty$, then
$$M(q) \geq \bigg(\frac{p}{p-1}-o(1)\bigg)\log_4 q.$$
\end{theorem}
We refer to \cref{3mod8} for a discussion on the case $q \equiv 3 \pmod{8}$.

There is a straightforward way to construct a Diophantine tuple over $\F_p$ explicitly using the least quadratic non-residue modulo $p$. Indeed, if we use $n(p)$ to denote the least quadratic non-residue modulo $p$, then $M(p) \geq \sqrt{n(p)}+1$ since $A=\{0,1,\ldots, \lfloor\sqrt{n(p)} \rfloor\}$ is a Diophantine tuple over $\F_p$. However, this construction does not lead to an improved lower bound on $M(p)$. In fact, under the generalized Riemann hypothesis (GRH), Lamzouri, Li, and Soundararajan \cite{LLS15} showed that $n(p)\leq (\log p)^2$. On the other hand, the best known lower bound is  $n_p=\Omega(\log p \log \log p)$ under GRH, due to Montgomery \cite{M71}. Thus, this construction only yields that $M(p)=\Omega(\sqrt{\log p \log \log p})$, which is worse than the trivial lower bound. Our construction is of a similar flavor, and we use a simple trick to convert products into sums to use the space of intervals more effectively so that the trivial lower bound on $M(q)$ can be improved. The essential idea of our proof is related to the properties of cyclotomic polynomials, which partially explains the factor $\frac{p}{p-1}$.

\section{Preliminaries}
In this section, we introduce two main tools of our construction: character sums and cyclotomic polynomials.

\subsection{Weil's bound and its consequences}

We first recall Weil's bound for complete character sums; see for example \cite[Theorem 5.41]{LN97}.

\begin{lemma}(Weil's bound)\label{Weil}
Let $\chi$ be a multiplicative character of $\F_q$ of order $d>1$, and let $g \in \F_q[x]$ be a monic polynomial of positive degree that is not an $d$-th power of a polynomial. 
Let $n$ be the number of distinct roots of $g$ in its
splitting field over $\F_q$. Then for any $a \in \F_q$,
$$\bigg |\sum_{x\in\mathbb{F}_q}\chi\big(ag(x)\big) \bigg|\le(n-1)\sqrt q \,.$$
\end{lemma}

The following lemma is a classical application of Weil's bound for complete character sums; see for example \cite[Exercise 5.66]{LN97}. 

\begin{lemma}\label{lem:countingsolns}
Let $d\geq 2$ be a positive integer. Let $q\equiv 1\pmod{d}$ be a prime power. Let $\chi$ be a multiplicative character of $\F_q$ with order $d$. Let $a_1,a_2,\ldots, a_k$ be $k$ distinct elements of $\F_q$, and let $\epsilon_1,\ldots, \epsilon_k \in \C$ be $d$-th roots of unity. Let $N$ denote the number of $x \in \F_q$ such that $\chi(x+a_i)=\epsilon_i$ for $1 \leq i \leq k$. Then 
$$
\bigg|N-\frac{q}{d^k}\bigg|\leq \bigg(k-1-\frac{k}{d}+\frac{1}{d^k}\bigg)\sqrt{q}+\frac{k}{d}.
$$
\end{lemma}

In the following proposition, we give a lower bound on the size of maximal Diophantine tuples over $\F_q$.  

\begin{proposition}\label{prop:maximal}
If $A$ is a maximal Diophantine $m$-tuple over $\F_q$, then  $q<2^{2m-2}m^{2}$. In particular, $M(q) \geq (1-o(1))\log_4 q$.     
\end{proposition}
\begin{proof}
Let $A=\{a_1, a_2, \ldots, a_m\}$. Since $A$ is maximal, it cannot be further extended. This means that 
$$
\emptyset=\{x \in \F_q^* \setminus A: a_ix+1 \text{ is a square}\} \supset \{x \in \F_q^* \setminus A: \chi(x+a_i^{-1})=\chi(a_i)\}:=B,
$$
where $\chi$ is the quadratic character of $\F_q$. If $q \geq 2^{2m-2}m^{2}$,  then \cref{lem:countingsolns} implies that
$$
|B| \geq \frac{q}{2^m}-\Big(\frac{m-2}{2}+\frac{1}{2^m}\Big)\sqrt{q}-\frac{3m}{2}>\frac{q}{2^m}-\frac{m}{2}\sqrt{q}\geq 0,
$$
a contradiction. Therefore, we must have $q<2^{2m-2}m^2$. It follows that $m \geq (1-o(1))\log_4 q$.
\end{proof}

\subsection{Cyclotomic polynomials}
Let $n$ be a positive integer. Let $\Phi_n(x)$ be the $n$-th {\em cyclotomic polynomial}, that is, the
monic polynomial whose roots are the $n$-th primitive roots of unity (in the complex number). We refer to the basic properties of cyclotomic polynomials in \cite{S22}. It is known that $\Phi_n(x) \in \Z[x]$ and the degree of $\Phi_n(x)$ is given by the Euler's totient function $\phi(n)$. Another important identity is 
$$
x^n-1=\prod_{d \mid n} \Phi_d(x).
$$

For our purpose, in the following discussions, we always view $\Phi_n(x)$ as a polynomial over $\F_q$. The following lemma is crucial in our proof; we provide a short proof for the sake of completeness.

\begin{lemma}\label{lem: Phi}
Let $n,n'$ be positive integers. 
\begin{enumerate}
    \item[\textup{(1)}] If $p \mid n$, then $\Phi_n(x)$ is a square of a polynomial. 
    \item[\textup{(2)}] If $p \nmid n$, then $\Phi_n(x)$ is square-free.
    \item[\textup{(3)}] If  $n \neq n'$ such that $p \nmid n$ and $p \nmid n'$, then $\Phi_n(x)$ and $\Phi_{n'}(x)$ are coprime. 
\end{enumerate}
\end{lemma}

\begin{proof}
\begin{enumerate}
    \item If $n=p^m r$ with $p \nmid r$ and $m \geq 1$, then we know $\Phi _{n}(x)=\Phi _{pr}(x^{p^{m-1}})$ and $\Phi_{pr}(x)=\Phi _{r}(x^{p})/\Phi_{r}(x)=(\Phi_{r}(x))^{p-1}$ from basic properties of cyclotomic polynomials \cite[Section 1]{S22}. It follows that $$\Phi _{n}(x)=\Phi _{pr}(x^{p^{m-1}})=(\Phi_{pr}(x))^{p^{m-1}}=(\Phi_{r}(x))^{(p-1)p^{m-1}}.$$ Thus  $\Phi_n(x)$ is a square of a polynomial.
    \item If $p \nmid n$, then $x^n-1$ is square-free since it does not have double root over $\overline{\F_q}$ by checking the derivative of $x^n-1$. Since $\Phi_n(x) \mid (x^n-1)$, it follows that $\Phi_n(x)$ is square-free.
    \item Consider $x^{nn'}-1$. We have $\Phi_n(x)\Phi_{n'}(x) \mid (x^{nn'}-1)$ by the definition of cyclotomic polynomials. Since $x^{nn'}-1$ is square-free, it follows that $\Phi_n(x)$ and $\Phi_{n'}(x)$ are coprime.
\end{enumerate}    
\end{proof}

\section{Explicit constructions of Diophantine tuples over finite fields}

In this section, we construct Diophantine $m$-tuples over $\mathbb{F}_{q}$ with $m \ge (1-o(1))\frac{p}{p-1}\log_{4}q$, where $q=p^{n}$. 
When $q=p$ is a prime, a lower bound of the largest size of Diophantine tuples in $\mathbb{F}_p$ is almost the same as $\log_{4}p$.
When $q=p^{n}$, this improves the lower bound $\log_4 q$ by a multiplicative constant.
We prove this in two cases depending on the condition of $q$.

\subsection{Case I: $q \equiv 1 \pmod{4}$}

Note that the condition $q \equiv 1 \pmod{4}$ implies that there is $r \in \mathbb{F}_{q}$ such that $r^{2}=-1$.

\begin{proposition}\label{prop: lower bound1}
Let $m$ be a positive even integer. 
Assume that $y \in \mathbb{F}_{q}^*$ is a square of order $\ge m$ such that $y^{i}-1$ a square for each $1 \le i \le m$. Let $r \in \F_q$ such that $r^{2}=-1$.
Then the set $$A=\{ry^{-m/2},ry^{-m/2+1},\ldots,ry^{-1},r,ry,ry^2,\ldots , ry^{m/2}\}$$ is a Diophantine $(m+1)$-tuple over $\mathbb{F}_{q}$. In particular, $M(q) \ge m+1$.
\end{proposition}

\begin{proof}
It suffices to show that $r^{2}y^{i}+1=r^{2}(y^{i}-1)$ is a square for each $|i| \le m-1$. 
This follows from the assumption immediately when $i$ is nonnegative.
When $i=-j$ is negative, we have $y^{-j}-1=(1-y^{j})/y^{j}$, which is a square since $y$ and $-1$ are squares.
This completes the proof.
\end{proof}

\begin{corollary}
If $q \equiv 1 \pmod{4}$, then
\[M(q) \geq \left \lfloor Q \right \rfloor, \quad \text{where } Q=\frac{p}{p-1}\bigg(\frac{\frac{1}{2}\log q -2\log\log q}{\log 2} +1\bigg).
\]
\end{corollary}
\begin{proof}
Cohen \cite[Theorem 1]{88Cohen} showed that there is $y \in \F_q^*$ satisfying the assumption of \cref{prop: lower bound1} with order at least $m=\lfloor Q\rfloor$ (see the first paragraph of \cite[Section 2]{88Cohen}), except that $m$ is possibly odd, in which case we replace $m$ by $m-1$. Thus, \cref{prop: lower bound1} implies that $M(q)\geq (m-1)+1=m=\left \lfloor Q \right \rfloor$.   
\end{proof}

\subsection{Case II: $q \equiv 7 \pmod{8}$} Since $q \equiv 7 \pmod 8$, it follows that $-1$ is not a square and the construction in \cref{prop: lower bound1} fails to work. Nevertheless, we take advantage of the fact that $2$ is a square. To get the desired lower bound on $M(q)$, we first establish the following new observation and then modify Cohen's proof. 

\begin{proposition}
\label{complete-condition}
Let $m$ be a positive even integer. 
Assume that $y \in \mathbb{F}_{q}^{*}$ is a square of order $\ge m$ such that $\Phi_{i}(y)$ is a square for each $i=2,4,\ldots,2m-2$ with $p \nmid i$.
Then the set 
$$A=\{ y^{-m/2},y^{-m/2+1},\ldots,y^{-1},1,y,\ldots,y^{m/2}\}$$
is a Diophantine $(m+1)$-tuple. In particular, $M(q) \ge m+1$.
\end{proposition}

\begin{proof}
It suffices to show that $y^{i}+1$ is a square for each $|i| \le m-1$. This holds when $i=0$ since $2$ is a square given that $q \equiv 7 \pmod 8$.
This also holds when $i$ is positive since 
\[y^{i}+1=\frac{y^{2i}-1}{y^{i}-1}=\prod_{\substack{d\nmid i,\\d|2i}}\Phi_{d}(y) , \quad \textrm{where $1\leq i \leq m-1$},\]
which is again a square by assumption. It is worthwhile to mention here that if $p \mid d$, then $\Phi_d(y)$ is automatically a square by \cref{lem: Phi}.
If $i$ is negative, we can write $i=-j$ and we have $y^{-j}+1=(y^{j}+1)/y^{j}$ where $1\leq j \leq m-1$, which is a square since 
$y$ is a square. 
Hence $A$ is a Diophantine $(m+1)$-tuple over $\F_{q}$.
\end{proof}

In order to get a lower bound of $M(q)$ with respect to $q$, we need to count the number of $y$ in the assumption of \cref{complete-condition}. 
The following observation helps to describe the counting problem using the cyclotomic polynomials $\Phi_{i}(y)$. 

\begin{lemma}
\label{Estimation-S_p(t)}
Let 
\[  S_{p}(m)=\sum_{\substack{i=1\\ p\nmid i }}^{m-1} \varphi(2i) \quad \text{and} \quad \text{$T_{p}(m)=m- \left \lfloor \frac{m-1}{p} \right \rfloor$}.\]
Then, for $m\geq 2$, we have
\begin{equation*}
S_{p}(m)<\frac{p}{p-1}T_{p}^{2}(m).   
\end{equation*}
\end{lemma}

\begin{proof}
Let 
$s_{p}(m)=\sum_{i=1, p\nmid i }^{m-1} \varphi(i)$.
The result follows from the basic properties of Euler's totient function $\varphi$-function. We have
\begin{align*}
S_{p}(m)=\sum_{\substack{1 \leq i \leq m-1\\ i \text{ even, } p\nmid i }} \varphi(2i)+\sum_{\substack{1 \leq i \leq m-1\\ i \text{ odd, } p\nmid i }} \varphi(2i)
 \le 2s_{p}(m)
 \le \frac{p}{p-1}T_{p}^{2}(m),
\end{align*}
where the last inequality is from \cite[Lemma 2]{88Cohen}.
\end{proof}
Hence, \cref{Estimation-S_p(t)} gives an upper bound of $S_{p}(m)$ using $T_{p}(m)$.
For our purpose, we define a set $I(m)=\{1 \leq i \leq  m-1: p \nmid i\} \cup \{0\}$.
Then we have that 
$T_p(m)=|I(m)|$.

Now, we are ready to prove our main theorem.

\begin{theorem}
Let $q=p^n>7$. Then 
\begin{equation}
\label{cohen-bound}
M(q) \geq \lfloor Q \rfloor, \quad \text{where } Q=\frac{p}{p-1}\bigg(\frac{\frac{1}{2}\log q -2\log\log q}{\log 2} +2\bigg).
\end{equation}
\end{theorem}
\begin{proof}
Without loss of generality, we assume $Q>2$. We separate the integral part and the non-integral of $Q$ as follows
\[Q=pu+v+\epsilon,\]
where $u$ and $v$ are non-negative integers with $0\leq v \leq p-1$ and $0\leq \epsilon <1$. Let $m=\floor{Q}=pu+v$ be the integral part of $Q$. 
Then we have
\[
T_{p}(m)= m - \left \lfloor \frac{m-1}{p} \right \rfloor \le (\log 2)^{-1}\left(\frac{1}{2}\log q -2 \log \log q \right)+2\]
by the same idea from \cite[Theorem 1]{88Cohen}.
From now on, we denote $T_{p}(m)$ by $T$.
Together with \cref{Estimation-S_p(t)}, we have
\begin{equation}\label{eq3}
T \leq (\log 2)^{-1}\left(\frac{1}{2}\log q -2 \log\log q\right)+2 \quad \text{and} \quad S_{p}(m)<\frac{p}{p-1}T^{2}, \quad (m\geq 2).   
\end{equation}

Let $\chi$ denote the quadratic character in $\F_q$.
Thus, we formulate the number of $y$, denoted by $N(m)$, which satisfies the conditions of \cref{complete-condition} as follows: 
\begin{equation}
\label{Nt}
N(m):=2^{-T} \sum_{\substack{y\in\F_q \\ \ord y \geq m }} \prod_{i \in I(m) }(1+\chi (\Phi_{2i}(y))),
\end{equation}
\textcolor{black}{where we denote $\Phi_{0}(y)=y$. The introduction of this notation allows us to replace the condition that $y$ is a square with the condition that $\Phi_{0}(y)$ is a square and write the above equation in a more compact form.} 
From \cref{Nt}, we deduce that
\begin{equation}\label{Nm}
N(m)  =2^{-T} \sum_{\substack{y\in\F_q \\ \ord y \geq m }} \left(1+\sum_{f}\chi (f(y))\right), 
\end{equation}
where the sum is over all $2^{T}-1$ non-constant polynomials $f$ of the form
\begin{equation}
    \label{product-of-cyclotomic}
    f(x)= \prod_{i \in I(m) } \Phi_{2i}^{k_{i}}(x), \quad  k_{i}=0 \text{~or~} 1.
\end{equation}
Now we rewrite equation \cref{Nm} into complete character sums:
\begin{align}
N(m) & =2^{-T} \sum_{\substack{y\in\F_q \\ \ord y \geq m }} \left(1+\sum_{f}\chi (f(y))\right) \notag\\
&=2^{-T} \sum_{\substack{y\in\F_q}} \left(1+\sum_{f}\chi (f(y))\right) - 2^{-T} \sum_{\substack{y\in\F_q \\ \ord y < m }} \left(1+\sum_{f}\chi (f(y))\right)\notag \\
& \geq 2^{-T} \left( q - \sum_{f}\left| \sum_{y\in \F_q}\chi(f(y))\right| \right)-S_p(m),\label{eq1}
\end{align}
where the last inequality can be deduced from
\[2^{-T} \sum_{\substack{y\in\F_q \\ \ord y < m }} \left(1+\sum_{f}\chi (f(y))\right) \leq s_{p}(m)+\frac{2^{T} - 1}{2^{T}} \leq S_p(m). \]

 Note that a polynomial $f$ of the form \cref{product-of-cyclotomic} is monic and is not a square of a polynomial by \cref{lem: Phi}. Thus, Weil's bound (\cref{Weil}) implies that
\[ \left| \sum_{y\in \F_q}\chi(f(y))\right| \leq (\deg f -1) \sqrt{q} \le \left(  \sum_{i \in I(m)} k_{i}\varphi(2i) -1  \right) \sqrt{q}. \]
Here, the second inequality is from the expression \cref{product-of-cyclotomic} of $f$.
Hence, since $\Phi_{2i}$ appears as a factor in $2^{T-1}$ of the polynomials $f$, we can write
\[
    \sum_{f}\left| \sum_{y\in \F_q}\chi(f(y))\right| \leq \left(2^{T-1} S_{p}(m)-(2^{T}-1)\right) \sqrt{q}
\]
Therefore, estimation \cref{eq1} and \cref{Estimation-S_p(t)} imply that
\begin{align}
    N(m)& \geq 2^{-T} \left( q- (2^{T-1} S_{p}(m)-(2^{T}-1))\sqrt{q}  \right) -S_p(m), \notag \\
        & \geq 2^{-T} \left( q- (2^{T-1} \frac{p}{p-1} T^2-(2^{T}-1))\sqrt{q}  \right) -\frac{p}{p-1} T^2, \notag \\
    & \geq 2^{-T} \sqrt{q}\left( \sqrt{q}- 2^{T-1} \frac{p}{p-1} T^2 \right)+(1-2^{-T})\sqrt{q}  -\frac{p}{p-1} T^2, \notag 
\end{align}
as $S_p(m)<\frac{p}{p-1} T^2$. If $T=1$ and $q > 4$, we get $N(m)>0$. Otherwise, if $\sqrt{q}>2^{T-1}\frac{p}{p-1} T^2$, we have $N(m)>0$ since
\[
(1-2^{-T})\sqrt{q}  -\frac{p}{p-1} T^2  >  (1-2^{-T})2^{T-1}\frac{p}{p-1} T^2  -\frac{p}{p-1} T^2>0\\
\]
where the last inequality follows since $T\geq 2$. 
Hence, in order to prove $N(m)>0$, it is sufficient to prove 
$\sqrt{q}>2^{T-1}\frac{p}{p-1} T^2$.
Hence, it suffices to show that
\[(T-1)\log 2 + 2\log T+\log\left( \frac{p}{p-1}\right) <\frac{1}{2}\log q.\]
Note that by inequality \cref{eq3},
\[\log T < \log \left( \frac{\frac{1}{2}\log q -2 \log \log q + 2 \log 2 }{\log 2}  \right) < \log \left( \frac{\log q }{2\log 2} \right)\]
for $q>e^2$.
Thus, we have $\log T < \log\log q  -\log (2\log 2)$, and by inequality \cref{eq3}, 
\begin{align*}
(T-1)\log 2 + 2\log T + \log\left( \frac{p}{p-1}\right) 
& \le \frac{1}{2}\log q -2 \log \log q + \log 2 +2 \log T +\log\left(\frac{p}{p-1}\right) \\
&<\frac{1}{2}\log q + \log 2 -2 \log (2\log 2)+\log\left(\frac{p}{p-1}\right)\\
&< \frac{1}{2} \log q,
\end{align*}
where the last inequality follows since $\log \left(\frac{p}{p-1}\right)<\log 2$ for any prime $p>2$. The above argument shows that there is $y \in \F_q^*$ satisfying the assumption of \cref{complete-condition} with order at least $m$, except that $m$ is possibly odd, in which case we may replace $m$ by $m-1$. \cref{complete-condition} then implies that 
\[M(q)\geq (m-1)+1=m=\left \lfloor Q\right \rfloor.\]
\end{proof}

\begin{remark}\label{3mod8}
In the case of $q \equiv 3 \pmod{8}$, the construction in \cref{complete-condition} almost works, except that $2$ is not a square. Nevertheless, by dropping those elements with negative exponents, we can still explicitly construct a Diophantine tuple in $\F_q$ with size at least $(1-o(1))\frac{p}{2(p-1)} \log_4 q$. Such a Diophantine tuple might not be maximal in view of \cref{prop:maximal}. 

We also attempted to use other sorts of constructions, but we did not make much progress. It seems our proof relies on the fact that a polynomial of the form $x^i\pm 1$ has a nice factorization into cyclotomic polynomials.
\end{remark}

\section*{Acknowledgements}
The second author thanks Greg Martin for helpful discussions. The third author was supported by the KIAS Individual Grant (CG082701) at the Korea Institute for Advanced Study and supported by the Institute for Basic Science (IBS-R029-C1).

\section*{Statements and Declarations}

The authors declare that they have no known competing financial interests or personal relationships that could have appeared to influence the work reported in this paper.

\bibliographystyle{abbrv}
\bibliography{references}

\end{document}